\documentclass[a4paper,leqno,10pt,twoside]{amsart}
\usepackage{amsthm}
\usepackage{amsmath}
\usepackage{amssymb}
\usepackage{amsfonts}
\usepackage{amscd}
\usepackage[english]{babel}
\usepackage{stmaryrd}
\usepackage{enumerate}

\newtheorem{thm}{Theorem}
\newtheorem{cor}{Corollary}
\newtheorem{prop}{Proposition}
\newtheorem{rem}{Remark}

\newtheorem*{defi}{Definition}
\newtheorem{lem}{Lemma}

\newtheorem*{thm2}{Theorem}

\begin{document}

    \title{Similar powers of a matrix}
    \author{Gerald BOURGEOIS}
    
    \address{G\'erald Bourgeois, \emph{GAATI, Universit\'e de la polyn\'esie fran\c caise, BP 6570, 98702 FAA'A, Tahiti, Polyn\'esie Fran\c caise.}}
    \email{bourgeois.gerald@gmail.com}
        

\begin{abstract}  Let $p,q$ be coprime integers such that $|p|+|q|>2$. We characterize the matrices $A\in\mathcal{M}_n(\mathbb{C})$ such that $A^p$ and $A^q$ are similar. 
If $A$ is invertible, we prove that $A$ is a polynomial in $A^p$ and $A^q$. To achieve this, we study the matrix equation $B^{-1}A^pB=A^q$. We show that for such matrices, $B^{-1}AB$ and $A$ commute. When $A$ is diagonalizable, $A$ is a root of $I_n$ and $B^{-1}AB$ is a power of $A$. We explicitly solve the previous equation when $A$ has $n$ distinct eigenvalues or when $A$ has a sole eigenvalue. In the second part, we completely solve the $2\times{2}$ case of the more general matrix equation $A^{r}B^{s}A^{r'}B^{s'}=\pm{I}_2$. 
\end{abstract}

\maketitle
    \section{Introduction}
\indent In \cite{4}, the matrices $A\in\mathcal{M}_n(\mathbb{C})$ such that $A^p=A$, with $p\geq{3}$, are characterized. In particular, it is shown that such an equality holds if and only if $A^{\#}=A^{p-2}$ where $A^{\#}$ is the group inverse of $A$.\\
In \cite{2}, the authors deal with a $\{K,s+1\}$-potent matrix, that is, a matrix $A$ such that $KA^{s+1}K=A$, where $s\geq{1}$ and $K$ is an involutory matrix. Clearly $A^{s+1}$ and $A$ are similar. They prove the following result
\begin{thm2}  \cite[Theorem 5]{2} Let $A\in\mathcal{M}_n(\mathbb{C})$, $\{\lambda_1,\cdots,\lambda_t\}$ be the spectrum of $A$ and, for every $j\leq{t}$, let $P_j$ be the eigenprojection associated to $\lambda_j$. The matrix $A$ is $\{K,s+1\}$-potent if and only if $A^{(s+1)^2}=A$ and, for every $i\leq{t}$, there exists a unique $j\leq{t}$ such that $\lambda_i=\lambda_j^{s+1}$ and $P_i=KP_jK$.
 \end{thm2}
 More related results can be found in \cite{8,6,7} or in \cite{9}. \\  
  \indent  In the first part of this paper, we are interested in a more general problem: let $p,q$ be coprime integers such that $|p|+|q|>2$. 
     We consider the $n\times{n}$ complex matrices $A$ such that $A^p$ and $A^q$ are similar. Clearly, if $A^p$ and $A^q$ are similar, then these two matrices have the same spectrum. The converse is false as we can see it in the following example
    $$A=\mathrm{diag}(\begin{pmatrix}i&0\\0&i\end{pmatrix},\begin{pmatrix}-i&1\\0&-i\end{pmatrix},J_3)  $$
 where $J_3$ is the Jordan nilpotent block of dimension $3$. The matrices $A^3$ and $A^5$ have the same spectrum but are not similar. However $A^3$ and $A^7$ are similar. 
 Assume that $0<p<q$ and that $A^p$ and $A^q$ are similar. We can easily show that there exists $k\leq{n}$ such that $A$ is similar to $\mathrm{diag}(N,T)$ where $N\in\mathcal{M}_k(\mathbb{C})$ satisfies $N^p=0_{k}$ and $T\in{G}L_{n-k}(\mathbb{C})$ is such that $T^p$ and $T^q$ are similar. In the particular case where $A^p=A^q$, $A^p$ is diagonalizable. This conclusion is not true in general, as we can check it in the previous example.  \\     
     In Theorem \ref{main}, we characterize the matrices $A$ such that $A^p$ and $A^q$ are similar. In this case, every eigenvalue of $A$ is shown to be $0$ or a root of unity. In the case when $A$ is invertible, $A$ is a polynomial in $A^p$ or in $A^q$. More explicitly, we consider the matrix equation 
     \begin{equation} \label{simil} B^{-1}A^pB=A^q  \end{equation}
      where the $n\times{n}$ complex invertible matrices $A,B$ are to be determined.\\
       We show that the matrices $B^{-1}AB$ and $A$ commute. Moreover if $A$ is diagonalizable, then $A$ is a root of $I_n$ and $B^{-1}AB$ is a power of $A$. In Theorem \ref{nvp}, we completely solve Eq. (\ref{simil}) when $A$ has $n$ distinct eigenvalues. Finally, the case where $A$ has a sole eigenvalue is considered too.\\
   \indent  In the second part of the article, we have a look at the $2\times{2}$ case of a generalization of Eq. (\ref{simil}). Let $\epsilon\in\{-1,1\}$ and let $r,r',s,s'$ be given non-zero integers such that $\mathrm{gcd}(r,r')=1$ and $\mathrm{gcd}(s,s')=1$. We consider the matrix equation
\begin{equation}   \label{equation 1}  A^{r}B^{s}A^{r'}B^{s'}=\epsilon{I}_n    \end{equation}
       We reduce the problem to the case where $A$ and $B$ are in $SL_2(\mathbb{C})$ and are not simultaneously triangularizable. In Theorem \ref{casgen}, the solutions of Eq. (\ref{equation 1}) are given. Moreover, if $r\not={r'}$ and $s\not={s'}$, then $A$, $B$ and $A^rB^s$ are roots of $I_2$. The methods employed to solve Eq. (\ref{equation 1}) are essentially different from those employed to solve Eq. (\ref{simil}). \\\\
    \indent We introduce notations that will be used in the sequel of the article.\\
    \textbf{Notation}.
    $i)$ Denote by $\mathbb{N}$ the set of positive integers.\\ 
      $ii)$ If $A$ is a square complex matrix, then $\sigma(A)$ denotes the set of distinct eigenvalues of $A$ and $m\sigma(A)$ denotes the set of eigenvalues of $A$ with multiplicities.\\
    $iii)$ A multiset $U$ is a set of complex numbers with multiplicities. Moreover, if $p\in\mathbb{Z}$, $U^p$ denotes the multiset of the $p$-th powers of the elements from $U$ with multiplicities.\\
    
  \section{Similarity of $A^p$ and $A^q$}
\addtocounter{equation}{-2}
 We consider the matrix equation
   \begin{equation}   B^{-1}A^pB=A^q \end{equation}
\addtocounter{equation}{1}
 when $A,B\in{G}L_n(\mathbb{C})$ and $p,q$ are given coprime integers. To avoid trivial situations, we assume that $|p|+|q|>2$. In the sequel, we put $C=B^{-1}AB$.
     \begin{prop}   \label{root} 
 Let $U$ be some finite multiset of non-zero complex numbers. Assume that $U^p$ and $U^q$ are equal. The following assertions hold\\
 $i)\;$ Elements of $U$ are roots of unity. Moreover their orders are coprime to $pq$.\\
 $ii)\;$ If $\lambda,\mu\in{U}$ are such that $\lambda^p=\mu^p$ or $\lambda^q=\mu^q$, then $\lambda=\mu$.\\
 $iii)$ For every $\lambda_1\in{U}$ there exists a unique sequence $(\lambda_i)_{i\in\mathbb{Z}}$ of elements of $U$ with, for every $i\in\mathbb{Z}$, $\lambda_i^q=\lambda_{i+1}^p$.\\
$iv)$ The application $(k,\lambda_1)\in\mathbb{Z}\times{U}\rightarrow\lambda_{k+1}\in{U}$ is an action of $\mathbb{Z}$ on $U$. The order of each element of the orbit $O_{\lambda_1}$ of $\lambda_1$ divides $p^{r_1}-q^{r_1}$ where $r_1=\mathrm{card}(O_{\lambda_1})$.
 \end{prop}
 \begin{proof}
$i)\;$ Let $\lambda_1\in U$. There exist $\lambda_2\in{U}$ such that $\lambda_1^q=\lambda_2^p$ and $\lambda_0\in{U}$ such that $\lambda_0^q=\lambda_1^p$. Inductively we can build a sequence $(\lambda_i)_{i\in\mathbb{Z}}$ in  $U$ with $\lambda_i^q=\lambda_{i+1}^p$ for all $i\in\mathbb{Z}$. Since $U$ is finite, there exist positive integers $u,v$ ($u<v$) such that $\lambda_u=\lambda_v$ and such that, for $u_1<v_1<v$, one has $\lambda_{u_1}\not=\lambda_{v_1}$. Since $\lambda_1^{q^{u-1}}=\lambda_u^{p^{u-1}}$ and $\lambda_1^{q^{v-1}}=\lambda_v^{p^{v-1}}$, we obtain that
$$
\lambda_1^{q^{u-1} p^{v-u}}=\lambda_u^{p^{v-1}}=\lambda_1^{q^{v-1}}.
$$
Hence, $\lambda_1$ is a root of unity of order $k$ dividing $q^{u-1}(p^{v-u}-q^{v-u})$. Thus $k$ and $p$ are coprime. In the same way, $k$ and $q$ are coprime. Finally $k$ divides $p^{v-u}-q^{v-u}$.\\
$ii)\;$ Now if $\lambda\not=\mu$ and $\lambda^q=\mu^q$, we have $(\lambda/\mu)^q=1$. Therefore the order of $\lambda/\mu$ is a divisor of $q$ which is $>1$. On the other hand, the order of $\lambda/\mu$ divides the least common multiple of the orders of $\lambda$ and $\mu$, and this number is coprime to $pq$. Hence, we obtain a contradiction. Therefore $\lambda=\mu$.\\
$iii)$ By $i)$ and $ii)$, the sequence $(\lambda_i)_{i\in\mathbb{Z}}$ is well defined.\\
$iv)$ Since $\lambda_1^{p^{v-u}}=\lambda_1^{q^{v-u}}=\lambda_{1+v-u}^{p^{v-u}}$, we have $\lambda_{1+v-u}=\lambda_1$ and $u=1$. We deduce that  $\{\lambda_1,\cdots,\lambda_{v-1}\}$ is the orbit of $\lambda_1$.  
  \end{proof}
  \begin{rem}  \label{orb}
  $i)$ When $U$ contains $r$ copies of the element $\lambda_1$, we associate $r$ orbits equal to $O_{\lambda_1}$. The elements of $U$ that are in the same orbit have the same order of multiplicity.\\
  $ii)$  Let $V$ be the set of distinct elements of $U$. The action of $\mathbb{Z}$ induces a permutation $\pi$ of $V$. If $\delta$ is the $\mathrm{lcm}$ of the cardinalities of the orbits, then the order of each element of $U$ divides $p^{\delta}-q^{\delta}$ and $\delta$ is the order of $\pi$.    
  \end{rem}
   \begin{lem}
  Let $N,M\in\mathcal{M}_n(\mathbb{C})$ be nilpotent matrices and $\lambda,\mu\in\mathbb{C}^*$. Then $(\lambda{I}_n+N)^p$ and $(\mu{I}_n+M)^q$ are similar if and only if $\lambda^p=\mu^q$ and the matrices $N,M$ are similar.  
  \end{lem}
    \begin{proof}
 An easy calculation gives 
$$(\lambda{I}_n+N)^p=\lambda^pI_n+p\lambda^{p-1}NZ$$
 where $Z$ is an invertible matrix such that $ZN=NZ$. Thus for every $k\in\mathbb{N}$, $$\ker(((\lambda{I}_n+N)^p-\lambda^pI_n)^k)=\ker(N^k).$$
  If $M,N$ are similar, then, for every $k\in\mathbb{N}$, $\dim(\ker(M^k))=\dim(\ker(N^k))$ and the characteristic spaces of $(\lambda{I}_n+N)^p$ and $(\mu{I}_n+M)^q$ have the same dimension.  
  \end{proof}
    \textbf{Notation.} Let $A\in{G}L_n(\mathbb{C})$.\\
  $i)$ Its decomposition in Jordan normal form can be written $$A=P\;\mathrm{diag}(\lambda_1{I}_{i_1}+N_1,\cdots,\lambda_m{I}_{i_m}+N_m)P^{-1}$$
   where $P$ is an invertible matrix,  $\sigma(A)=(\lambda_k)_{k\leq{m}}$ (the $(\lambda_k)_{k\leq{m}}$ are pairwise distinct and $i_k$ denotes the multiplicity of $\lambda_k$) and $(N_k)_{k\leq{m}}$ are Jordan nilpotent matrices.\\
  $ii)$ If moreover $(m\sigma(A))^p=(m\sigma(A))^q$, by Proposition \ref{root}. $iv)$, for $U=m\sigma(A)$, we can associate the set of distinct orbits $(O_j)_{j\leq{\tau}}$ and their multiplicities $(r_j)_{j\leq{\tau}}$.\\\\
  From the previous results, we deduce easily 
  \begin{thm}   \label{main}
  Let $A\in\mathcal{M}_n(\mathbb{C})$ such that $(m\sigma(A))^p=(m\sigma(A))^q$.\\
  $i)$ When $A$ is invertible, the matrices $A^p$ and $A^q$ are similar if and only if for every $j\leq{\tau}$, $\lambda_k,\lambda_l\in{O}_j$, one has $i_k=i_l=r_j$ and $N_k=N_l$ up to orderings of the Jordan nilpotent blocks contained in $N_k$ and $N_l$.\\
  $ii)$ Suppose that $A$ is not invertible and $1\leq{p}<q$. The matrices $A^p$ and $A^q$ are similar if and only if 
  $$A=P\;\mathrm{diag}(A',N)P^{-1}$$
   where $P$ is an arbitrary invertible matrix, $A'\in{G}L_{n-k}(\mathbb{C})$ satisfies the properties given in i) and $N\in\mathcal{M}_k(\mathbb{C})$ satisfies $N^p=0_k$.   
  \end{thm}
\begin{rem} \label{card} Let $A$ be such that $A^p$ and $A^q$ are similar. For every $\lambda\in\sigma{(}A)\setminus\{0\}$, there exist $t\in\llbracket{1},\mathrm{card}(\sigma(A))\rrbracket$ and $k$, a divisor of $p^t-q^t$, such that $\lambda$ is a root of unity of order $k$.
\end{rem}
     To solve Eq. (\ref{simil}), it remains to determine the matrix $B$. This can be achieved by solving the Sylvester homogeneous equation $A^pX-XA^q=0$ and extracting the invertible solutions. However a more conceptual solution can be given.   
   \begin{prop} \label{dur}
  Let $A,B\in{G}L_n(\mathbb{C})$ satisfy Eq. (\ref{simil}). Then $A$ is a polynomial in $A^q$ or in $A^p$. In particular, $B^{-1}AB$ and $A$ commute.  
  \end{prop}
  \begin{proof}
  Let $f:x\in{V}\rightarrow{x}^q$ where $V$ is a neighborhood of $\sigma(A)$. According to Proposition \ref{root}. $ii)$, $f$ is a holomorphic function that is one to one on $\sigma(A)$ and $f'\not=0$ on $\sigma(A)$. By \cite[Theorem 2]{1}, $A$ is a polynomial in $A^q$. By symmetry, $A$ is a polynomial in $A^p$. Since $C^p=A^q$, $C$ and $A^q$ commute as $C$ and $A$ do.  
  \end{proof}
  In the next two results, we assume that $A$ is diagonalizable. Note that we will see, in Remark \ref{nondiag}, that $A$ can be non-diagonalizable.
  \begin{defi}
  We say that $A\in\mathcal{M}_n(\mathbb{C})$ is a root of $I_n$ if there exists $k\in\mathbb{N}$ such that $A^k=I_n$.  
  \end{defi} 
  \begin{prop}   \label{diag} Let $A,B\in{G}L_n(\mathbb{C})$ satisfy Eq (\ref{simil}). If $A$ is diagonalizable, then $A$ is a root of $I_n$ and $B^{-1}AB$ is a power of $A$.  
  \end{prop}
  \begin{proof}
  Using Proposition \ref{root} with $U=\sigma(A)$, we deduce that the eigenvalues of $A$ are roots of unity such that their orders are coprime to $pq$. Since $A$ is diagonalizable, there exists $r\in\mathbb{N}$ such that $\mathrm{gcd}(r,pq)=1$ and $A^r=I_n$. Therefore there exist $\alpha,\beta\in\mathbb{Z}$ such that $\alpha{p}+\beta{r}=1$. We deduce that $C=(C^p)^{\alpha}(C^r)^\beta=A^{\alpha{q}}$.
    \end{proof} 
    Now we give a complete solution of Eq (\ref{simil}) in the unknowns $A,B$ when $A$ is assumed to have $n$ distinct eigenvalues.
  \begin{rem}  \label{distinct}
    Suppose that $B^{-1}A^pB=A^q$ and that the eigenvalues $(\lambda_i)_i$ of $A$ are non-zero and pairwise distinct. We may assume that $A=\mathrm{diag}(\lambda_1,\cdots,\lambda_n)$. By reordering the $(\lambda_i)_i$, we may assume that the permutation $\pi$, considered in Remark \ref{orb}. $ii)$, is the product of disjoint cycles $(C_k)_k$ in the form $$C_k=(\lambda_{i_k},\lambda_{i_k+1},\cdots,\lambda_{i_{k+1}-1}).$$
     Since $B^{-1}AB$ is diagonal, $B$ permutes the eigenspaces of $A$ and $$B=\mathrm{diag}(b_1,\cdots,b_n)\Sigma$$
     where the entries $(b_i)_{i\leq{n}}$ are complex numbers and $\Sigma$ is a permutation matrix. The condition $B^{-1}A^pB=A^q$ is equivalent to $$\Sigma^{-1}\mathrm{diag}(\lambda_1^p,\cdots,\lambda_n^p)\Sigma=\mathrm{diag}(\lambda_1^q,\cdots,\lambda_n^q).$$
     By Proposition \ref{root}. $ii)$, the $(\lambda_i^p)_i$ and the $(\lambda_i^q)_i$ are pairwise distinct. Therefore, $\Sigma$ is the permutation matrix associated to $\pi$. Thus the solutions in $B$ are  trivial and we may suppose that $\pi$ is the cycle $(1,2,\cdots,n)$.
    \end{rem}
    \textbf{Notation.}  Denote by $R$ the ring $\mathbb{Z}/(q^n-p^n)$.
    \begin{thm} \label{nvp} Let $A=diag(\lambda_1,\cdots,\lambda_n)$, with the non-zero and pairwise distinct $(\lambda_i)'s$, be such that $A^p$ and $A^q$ are similar. There exists
   $$k_1\in{R}\setminus\bigcup_{z|n,z<n}(\dfrac{q^n-p^n}{q^z-p^z})R$$ 
such that, for every $1\leq{u}\leq{n}$, $\lambda_u=\exp(\dfrac{2i\pi{k}_{u}}{q^n-p^n})$ where $k_{u}=(p^{-1}q)^{u-1}k_1$.   
    \end{thm}
    \begin{proof}
     We put $\lambda_{n+1}=\lambda_1$. According to Remark \ref{orb}. $ii)$, $$\sigma(A)\subset{V}_n=\{z\in\mathbb{C}\;|\;z^{q^n-p^n}=1\},$$
      that is, for all $u\leq{n}$, there exists $k_u\in{R}$ such that $$\lambda_u=\exp(\dfrac{2i\pi{k}_{u}}{q^n-p^n}).$$
       The condition to be fulfilled is: for all $u\leq{n}$, $\lambda_u^q=\lambda_{u+1}^p$ or $qk_u=pk_{u+1}$. Note that $p^{-1}$, the inverse of $p$ in $R$ exists since $\gcd(p,q)=1$. Thus $k_{u+1}=(p^{-1}q)k_u$ and $k_{u}=(p^{-1}q)^{u-1}k_1$. Hence $\sigma(A)$ is determined by the choice of $k_1$. Remark that $(p^{-1}q)^n=1$ and then $k_{n+1}=k_1$. \\ 
 \indent Moreover the $(k_u)_{u\leq{n}}$ must be distinct. We consider the following equation in $R$: $(p^{-1}q)^\alpha{k}=k$ where $\alpha\in\llbracket{1},n\llbracket$. Since $p,q$ are invertible in $R$, we obtain $k(p^{\alpha}-q^{\alpha})=0$. From 
 $$\gcd(q^n-p^n,p^{\alpha}-q^{\alpha})=q^z-p^z,$$
  where $z=\gcd(n,\alpha)$, we deduce that $p^{\alpha}-q^{\alpha}=(q^z-p^z)f$ where $f$ is invertible in $R$. Therefore
   $k\in(\dfrac{q^n-p^n}{q^z-p^z})R$. When $\alpha$ goes through $\{1,\cdots,n-1\}$, $z$ goes through the strict divisors of $n$. Finally the required condition on $k_1$ is\\
  $$\text{for every strict divisor }z\text{ of }n,\; k_1\notin(\dfrac{q^n-p^n}{q^z-p^z})R.$$ 
     \end{proof}     
        \begin{rem} \label{nondiag} There are invertible solutions of Eq. (\ref{simil}) such that $A$ is not diagonalizable and $C$ is not a power of $A$. Indeed, let $n=4,p=2,q=3$. If $A\in{G}L_4(\mathbb{C})$ is such that $A^2$ and $A^3$ are similar and if $\sigma{(}A)=\{\lambda,\mu\}$, then, by Remark \ref{card}, $\lambda^5=\mu^5=1$. A solution of Eq. (\ref{simil}) is $$A=\begin{pmatrix}\lambda&1&0&0\\0&\lambda&0&0\\0&0&\bar{\lambda}&1\\0&0&0&\bar{\lambda}\end{pmatrix},B=\begin{pmatrix}0&0&1&0\\0&0&0&\bar{\alpha}\\1&0&0&0\\0&\alpha&0&0\end{pmatrix}$$ 
 where $\lambda=e^{2i\pi/5},\alpha=\dfrac{3}{2}e^{6i\pi/5}$. Clearly, $C=\begin{pmatrix}\bar{\lambda}&\alpha&0&0\\0&\bar{\lambda}&0&0\\0&0&\lambda&\bar{\alpha}\\0&0&0&\lambda\end{pmatrix}$ is not a power of $A$.     
    \end{rem}
    We consider the case where $A$ has a sole eigenvalue.\\
    \textbf{Notation.} Let $S$ be a square matrix. We denote by $\mathcal{C}(S)$ the set of invertible matrices that commute with $S$, that is   $$\mathcal{C}(S)=\{\Delta\in{G}L_n(\mathbb{C})\;|\;\Delta{S}=S\Delta\}.$$
    \begin{rem} It should be noted that $\mathcal{C}(S)$ is not the commutant of $S$.
    \end{rem}
    \begin{prop} Let $A\in{G}L_n(\mathbb{C})$ such that $\sigma{(}A)=\{\lambda\}$ with $\lambda^{q-p}=1$ and let $N=A-\lambda{I}_n$. One has\\
    $i)$ $A^p$ and $A^q$ are similar.\\
    $ii)$ There exists a unique polynomial $P\in\mathbb{C}[X]$ with least degree, such that for every matrix $C$ defined above, $C-\lambda{I}_n=P(N)$.\\
    $iii)$ There exists a matrix $B_0$ such that $B_0^{-1}NB_0=P(N)$ and the set of matrices $B$ satisfying $B^{-1}A^pB=A^q$ is $\mathcal{C}(N)B_0$. 
    \end{prop}
    \begin{proof} $i)$ is clear.\\
    $ii)$ Necessarily $C$ is in the form $C=\lambda{I}_n+M$ where $M$ is a nilpotent matrix, similar to $N$ and such that $MN=NM$. Let $d$ be the nilpotence index of $N$ or $M$. The relation $C^p=A^q$ can be written
 \begin{eqnarray}  \label{dec} (\sum_{k=2}^{d}\binom{p}{k-1}\lambda^{p-k+1}M^{k-2})M=(\sum_{k=2}^{d}\binom{q}{k-1}\lambda^{q-k+1}N^{k-2})N, \end{eqnarray}
    that is $EM=FN$ where $E,F$ are invertible matrices that both commute with $M$ and $N$. According to Proposition \ref{dur}, $A$ is a polynomial in $A^q=C^p$. Therefore $N$ is a polynomial in $C^p$ and $N$ is a polynomial in $M$. Since $N$ and $M$ are nilpotent and $\mathrm{ker}(M)=\mathrm{ker}(N)$,  one has $N=\beta_1M+\beta_2M^2+\cdots+\beta_{d-1}M^{d-1}$ where $\beta_1\not=0$. By \cite[Theorem 2]{1}, we conclude that $M$ is a polynomial in $N$ in the form
     $$M=P(N)=\alpha_1N+\alpha_2N^2+\cdots+\alpha_{d-1}N^{d-1}$$
      where $\alpha_1\not=0$. Since $N,N^2,\cdots,N^{d-1}$ are linearly independent, we can determine the $(\alpha_i)_{i=1,\cdots,d-1}$. Finally we obtain a unique solution in $M$ as a function of $\lambda$ and $N$. \\
 $iii)$ It remains to determine the matrices $B$ such that 
   \begin{equation} \label{nilp} B^{-1}NB=M. \end{equation}
   First we build a particular solution of (\ref{nilp}).
  Let $\{e_1,\cdots,e_n\}$ be a basis in which $N$ is in Jordan form: $N=\mathrm{diag}(J_{r_1},\cdots,J_{r_k})$ where $J_{r_i}$ is the Jordan nilpotent block of dimension $r_i$. The matrix $M$ is a polynomial in $N$, hence $M=\mathrm{diag}(M_1,\cdots,M_k)$ where $M_i$ has dimension $r_i$. We look for a particular solution in the form $$B_0=\mathrm{diag}(B_1,\cdots,B_k)$$
   where, for every $i$, $B_i^{-1}J_{r_i}B_i=M_i$. Thus we may assume that $N=J_n$. The set $\mathcal{B}=\{M^{n-1}(e_n),M^{n-2}(e_n),\cdots,e_n\}$ is a basis of $\mathbb{C}^n$. Let $P$ be the $\mathcal{B}$-basis matrix. We have $M=PNP^{-1}$ and we can choose $B_0=P^{-1}$. Clearly the set of solutions of (\ref{nilp}) is $\mathcal{C}(N)B_0$.
       \end{proof}
   \begin{rem} 
  $i)$ A complete description of $\mathcal{C}(N)$ is given in \cite{3} or \cite{5}.\\
  $ii)$ Knowing $A$, all the computations can be performed explicitly.
   \end{rem}
   
    \section{On the equation $A^rB^sA^{r'}B^{s'}=\epsilon{I}_2$} 
\addtocounter{equation}{-3}
    Instead of Eq. (\ref{simil}), we consider the more general matrix equation  
    \begin{equation}A^{r}B^{s}A^{r'}B^{s'}=\epsilon{I}_2 .\end{equation}
\addtocounter{equation}{2}
    We restrict ourselves to $2\times{2}$ complex matrices. The integers $r,r'$ (resp. $s,s'$) are considered to be coprime and $\epsilon=\pm{1}$.
     \begin{defi} The complex matrices $A,B$ are said simultaneously triangularizable (denoted by $ST$) if there exists an invertible complex matrix $P$ such that $P^{-1}AP$ and $P^{-1}BP$ are upper triangular.
     \end{defi}
\begin {lem} \label {inv} We may assume that the matrices $A,B$ satisfy $\mathrm{det}(A)=\mathrm{det}(B)=1$.  
  \end {lem} 
  \begin{proof} Let $A,B$ be invertible matrices satisfying Eq (\ref{equation 1}). There exist $\lambda,\mu\in\mathbb{C}^*$ such that $A=\lambda{A}_1$ with $\mathrm{det}(A_1)=1$ and $B=\mu{B}_1$ with $\mathrm{det}(B_1)=1$. One has $\lambda^{2r+2r'}\mu^{2s+2s'}=1$, that is $\lambda^{r+r'}\mu^{s+s'}=\pm{1}$. Therefore $$A_1^{r}B_1^{s}A_1^{r'}B_1^{s'}=\lambda^{r+r'}\mu^{s+s'}(\epsilon{I}_2)=\pm\epsilon{I}_2.$$
  \end{proof}
  In the sequel, we assume that $\mathrm{det}(A)=\mathrm{det}(B)=1$. The case where $A$ and $B$ are ST is not interesting. Indeed we have
  \begin{prop}
     Suppose that $A,B$ are $ST$. The Eq. (\ref{equation 1}) can be reduced to a triangular system of polynomial equations in four unknowns.     
     \end{prop}
     \begin{proof}
    We may assume that $A,B$ are upper triangular matrices in the form $$A=\begin{pmatrix}u&v\\0&1/u\end{pmatrix},B=\begin{pmatrix}p&q\\0&1/p\end{pmatrix},$$
    where $(p,u,q,v)\in(\mathbb{C}^*)^2\times\mathbb{C}^2$.
     One checks easily that Eq. (\ref{equation 1}) is equivalent to a system in the form
      $$\left\{\begin{array}{cl}u^{r+r'}p^{s+s'}&=\epsilon\\
      v\phi(u,p)+q\psi(u,p)&=0\;,\end{array}\right.$$
      where $\phi,\psi$ are polynomials that depend on $r,r',s,s'$.
     \end{proof}
     In the sequel, we assume that $A,B$ are not $ST$, that is they have no common eigenvectors.
      \begin {lem} \label{sym}   Let $A,B\in\mathcal{M}_2(\mathbb{C})$ be not $ST$. They are simultaneously similar to two symmetric matrices.  
  \end {lem} 
  \begin{proof}
  $Case\;\; 1.$ $A$ is diagonalizable. We may asssume that $A$ is diagonal and $B=\begin{pmatrix}a&b\\c&d\end{pmatrix}$ with $bc\not=0$. Write $P=\begin{pmatrix}x&0\\0&y\end{pmatrix}$, with $x^2=b,y^2=c$. The matrices $P^{-1}AP$ and $P^{-1}BP$ are symmetric.\\
  $Case \;\;2.$ $A,B$ are not diagonalizable. One has $A=\lambda{I}_2+M,B=\mu{I}_2+N$ where $\lambda,\mu\in\mathbb{C}$ and $M,N$ are nilpotent matrices that are not $ST$. We can construct a basis of $\mathbb{C}^2$ containing an eigenvector of $M$ and an eigenvector of $N$. Thus we may assume that $$A=M=\begin{pmatrix}0&\alpha\\0&0\end{pmatrix}\text{\;and\;}
   B=N=\begin{pmatrix}0&0\\\beta&0\end{pmatrix},$$
    with $\alpha,\beta\in\mathbb{C}-\{0\}$. For $P=\begin{pmatrix}1&i\\1&-i\end{pmatrix}$, $P^{-1}AP$ and $P^{-1}BP$ are symmetric.
  \end{proof}
  \begin {cor} \label{orthog} Let $A,B\in\mathcal{M}_2(\mathbb{C})$ that are not $ST$. If $A^{r}B^{s}A^{r'}B^{s'}=\epsilon{I}_2$, then $A^{-r}B^{-s}A^{-r'}B^{-s'}=\epsilon{I}_2$.  
  \end {cor} 
  \begin{proof} By Lemma \ref{sym}, we may assume that $A,B$ are symmetric matrices. Easy computation allows to conclude the assertion.  
  \end{proof}
    \begin {prop} \label{rac} Let $A,B\in{S}L_2(\mathbb{C})$ that are not $ST$ and satisfy Eq. (\ref{equation 1}). Then
     $$(A^{r-r'}=\pm{I}_2 \text{ or } B^{s}=\pm{I}_2) \text{ and }(B^{s-s'}=\pm{I}_2 \text{ or } A^{r'}=\pm{I}_2).$$  
  \end {prop}
  \begin{proof}
 By Corollary \ref{orthog}, we find easily $A^{-r}B^{-s}A^{-r'}=A^{-r'}B^{-s}A^{-r}$, that is $A^{r-r'}B^{s}=B^{s}A^{r-r'}$. Since  $B^{s}$ and $A^{r-r'}$ commute, they have a common eigenvector. That is impossible but if $B^{s}=\pm{I}_2$ or $A^{r-r'}=\pm{I}_2$. By symmetry, the second relation holds. 
  \end{proof}
  \textbf{Notation.}
Let $k\in\mathbb{Z}$. Denote by $\phi_k$ the continuous function defined on $\mathbb{C}^*$ by
$$\phi_k(t)=\left\{\begin{array}{cl}\dfrac{1-t^{2k}}{t^{k-1}(1-t^2)}&\text{ if }t\not=\pm{1}\\
k&\text{ if }t=1\\
(-1)^{k-1}k&\text{ if }t=-1\text{\;\;.}\end{array}\right.$$
  \begin{lem} \label{decomp} Let $A,B\in{S}L_2(\mathbb{C})$ that are not $ST$. The following assertions hold\\
   i) The matrices $A,B$ are simultaneously similar to matrices in the form $$A_1=\begin{pmatrix}u&v\\0&u^{-1}\end{pmatrix}\;,\;B_1=\begin{pmatrix}p&0\\q&p^{-1}\end{pmatrix},$$
    where $u,v,p,q$ are non-zero complex numbers such that
    \begin{equation}   \label{inequal}  (p^2-1)(u^2-1)+uvpq\not=0.    \end{equation}  
   ii) For every $k\in\mathbb{Z}$, we have 
 $$A_1^k=\begin{pmatrix}u^k&v\phi_k(u)\\0&u^{-k}\end{pmatrix}\;,\; B_1^{k}=\begin{pmatrix}p^{k}&0\\q\phi_k(p)&p^{-k}\end{pmatrix}.$$
   iii) Let $\alpha\in\{1,-1\}$. The equality $A_1^k=\alpha{I}_2$ (resp. $B_1^k=\alpha{I}_2$) holds if and only if $u^k=\alpha$ and $u^2\not=1$ (resp. $p^k=\alpha$ and $p^2\not=1$).  
  \end{lem}
  \begin{proof}
 $i)$ There exists a basis of $\mathbb{C}^2$ containing an eigenvector of $A$ and an eigenvector of $B$. Such a change of basis transforms $A$ and $B$ in the form 
$$A_1=\begin{pmatrix}u&v\\0&u^{-1}\end{pmatrix}\;,\;B_1=\begin{pmatrix}p&0\\q&p^{-1}\end{pmatrix}$$
where $uvpq\not=0$. Moreover $A_1$ and $B_1$ are not $ST$ is equivalent to $\det([A_1,B_1])\not=0$, that is Relation (\ref{inequal}).
\\
 $ii)$ and $iii)$ are straightforward computations.
  \end{proof}
  \begin{lem} \label{elem}
   Let $A,B\in\mathcal{M}_2(\mathbb{C})$. The matrices $A,B$ have determinant $1$, are not $ST$ and $(AB)^2=-I_2$ if and only if $A,B$ are simultaneously similar to the matrices $A_1,B_1$ defined in Lemma \ref{decomp} with the condition
    \begin{equation}   \label{cond}  qv=\dfrac{-1-u^2p^2}{up},\text{ where } \left\{\begin{array}{l}up\not=0\\u^2+p^2\not=0\\1+u^2p^2\not=0\;.\end{array}\right.  \end{equation} 
  \end{lem}
  \begin{proof}
  This follows easily from Lemma \ref{decomp}. $i)$.
  \end{proof}
  \begin{prop}  \label{fond}
  Let $A,B\in{S}L_2(\mathbb{C})$ that are not $ST$. The matrices $A,B$ satisfy Eq. (\ref{equation 1}) if and only if there exists $\alpha\in\{-1,1\}$ such that $$A^{r-r'}=\alpha{I}_2,\;B^{s-s'}=-\alpha\epsilon{I}_2 \text{\;and\;} (A^{r}B^{s})^2=-I_2.$$
     \end{prop}
  \begin{proof}
  According to Proposition \ref{rac}, we have to consider the following cases.\\
  $Case\;\;1$. $A^{r-r'}=\pm{I}_2$ and $A^{r'}=\pm{I}_2$. This implies that $A^{r}=\pm{I}_2$. Since $\mathrm{gcd}(r,r')=1$, $A=\pm{I}_2$, that is impossible.\\
  $Case\;\;2$. $B^{s}=\pm{I}_2$ and $B^{s-s'}=\pm{I}_2$. By symmetry of case 1.\\
    $Case\;\; 3$. $B^{s}=\pm{I}_2$ and $A^{r'}=\pm{I}_2$. One has $A^{r}B^{s'}=\pm{I}_2$. Then, for instance, $A^{r}=\pm{I}_2$ and finally $A=\pm{I}_2$. This is impossible.\\
  $Case\;\; 4$. There exist $\alpha,\beta\in\{-1,1\}$ such that $A^{r-r'}=\alpha{I}_2$ and $B^{s-s'}=\beta{I}_2$.
  One has $(A^{r}B^{s})^2=\pm{I}_2$. 
  Assume $(A^{r}B^{s})^2=I_2$. The relations $\mathrm{det}(A)=\mathrm{det}(B)=1$ imply $A^{r}B^{s}=\pm{I}_2$. Then, for instance $A^r=\pm{I}_2$ and we conclude as in Case 3.
 Finally $(A^{r}B^{s})^2=-I_2$ and $\beta=-\alpha\epsilon$.\\
  The converse is clear.
  \end{proof} 
    \begin{cor} \label{imp}
  If $r-r'=\pm{1}$ or $s-s'=\pm{1}$, then there are no pairs $(A,B)\in{S}L_2(\mathbb{C})$ that are not $ST$ and satisfy Eq. (\ref{equation 1}). 
  \end{cor}
  \begin{proof}
 Assume, for instance, that $r-r'=\pm{1}$. According to Proposition \ref{fond}, necessarily one has $A=\pm{I}_2$, that is impossible. 
  \end{proof}
  \begin{thm} \label{casgen}
  Assume that $r-r'\not=\pm{1}$ and $s-s'\not=\pm{1}$. A pair $(A,B)\in{S}L_2^{\;\;2}(\mathbb{C})$ is not $ST$ and satisfies Eq. (\ref{equation 1}) if and only if it is conjugate to a pair in the form $$(\begin{pmatrix}u&v\\0&u^{-1}\end{pmatrix},\begin{pmatrix}p&0\\q&p^{-1}\end{pmatrix})$$
   where $u,v,p,q\in\mathbb{C}^*$ are such that
    $$u^{r-r'}=\pm{1}\;,\;p^{s-s'}=-\epsilon{u}^{r-r'},$$ 
  \begin{equation}  \label{ineg} u^{2r}+p^{2s}\not=0,1+u^{2r}p^{2s}\not=0\;\;\text{and} \end{equation}
  \begin{equation} \label{relat} qv=\dfrac{-1-u^{2r}p^{2s}}{u^{r}\phi_{r}(u)p^{s}\phi_{s}(p)}\;.\end{equation} 
  \end{thm}
  \begin{proof}
 According to Lemma \ref{decomp}, we may assume $A=\begin{pmatrix}u&v\\0&u^{-1}\end{pmatrix},B=\begin{pmatrix}p&0\\q&p^{-1}\end{pmatrix}$, where $uvpq\not=0$. For every $r\in\mathbb{Z}$, one has $A^{r}=\begin{pmatrix}u^{r}&v\phi_{r}(u)\\0&u^{-r}\end{pmatrix},B^{s}=\begin{pmatrix}p^{s}&0\\q\phi_{s}(p)&p^{-s}\end{pmatrix}$. According to Proposition \ref{fond}, $(A^{r}B^{s})^2=-I_2$. Suppose that $A^{r}$ and $B^{s}$ have a common eigenvector. Then, for instance,  $A^{r}=\pm{I}_2$, $A^{r'}=\pm{I}_2$ and finally $A=\pm{I}_2$. This is impossible. Thus $A^{r}$ and $B^{s}$ are not $ST$. By Lemma \ref{elem}, Conditions (\ref{ineg}) and (\ref{relat}) hold. By Proposition \ref{fond}, there exists $\alpha\in\{-1,1\}$ such that $A^{r-r'}=\alpha{I}_2,B^{s-s'}=-\alpha\epsilon{I}_2$. We deduce that $u^{r-r'}=\alpha$ and $p^{s-s'}=-\alpha\epsilon$.\\
 The converse is clear.   
  \end{proof}
  \begin{rem}
 Note that $u$ and $p$ are roots of unity except if $r=r'=\pm{1}$ or $s=s'=\pm{1}$.\\
   \end{rem}
   
    \textbf{Acknowledgements}. 
 The author thanks F. Luca for his participation in the proof of Proposition \ref{root} and D. Adam for many valuable discussions. The author thanks the referees for their valuable suggestions.

\bibliographystyle{plain}

\end{document}